\documentclass[11pt,twoside]{amsart}
\usepackage{epic}
\usepackage[latin1]{inputenc}
\usepackage[OT1]{fontenc}
\usepackage{amscd}
\usepackage{amsmath}
\usepackage{amsthm}
\usepackage{amssymb}
\usepackage[all]{xy}
\xyoption{curve}
\usepackage{ifthen} 
\usepackage{hyperref} 
\usepackage{graphicx}
\usepackage{multirow}

\newtheorem{theorem}{Theorem}[section]

\newtheorem{proposition}[theorem]{Proposition}

\newtheorem{lemma}[theorem]{Lemma}
\newtheorem{corollary}[theorem]{Corollary}
\newtheorem{conjecture}[theorem]{Conjecture}

\numberwithin{equation}{section}

\theoremstyle{definition}
\newtheorem{definition}[theorem]{Definition}
\newtheorem{remark}[theorem]{Remark}

\newtheorem{notation}[theorem]{Notation}

\newcommand{\Hom}{{\rm Hom}}

\newcommand{\CC}{\mathbb{C}}

\newcommand{\FF}{\mathbb{F}}

\newcommand{\PP}{\mathbb{P}}

\newcommand{\ZZ}{\mathbb{Z}}

\newcommand{\sL}{\mathcal{L}}
\newcommand{\sO}{\mathcal{O}}
\newcommand{\sD}{\mathcal{D}}

\renewcommand{\to}{\xymatrix@1@=15pt{\ar[r]&}}
\renewcommand{\rightarrow}{\xymatrix@1@=15pt{\ar[r]&}}
\renewcommand{\mapsto}{\xymatrix@1@=15pt{\ar@{|->}[r]&}}
\renewcommand{\twoheadrightarrow}{\xymatrix@1@=15pt{\ar@{->>}[r]&}}
\renewcommand{\hookrightarrow}{\xymatrix@1@=15pt{\ar@{^(->}[r]&}}
\newcommand{\congpf}{\xymatrix@1@=15pt{\ar[r]^-\sim&}}
\renewcommand{\cong}{\simeq}

\DeclareMathOperator{\Pic}{Pic}
\DeclareMathOperator{\rad}{rad}
\DeclareMathOperator{\degree}{deg}

\newcommand{\tildeD}{\widetilde{D}}

\newcommand{\tildeX}{\widetilde{X}}
\newcommand{\tildeE}{\widetilde{E}}
\newcommand{\tildeL}{\widetilde{L}}
\newcommand{\tildeC}{\widetilde{C}}
\newcommand{\tildeK}{\widetilde{K}}
\newcommand{\tildeB}{\widetilde{B}}
\newcommand{\tildeF}{\widetilde{F}}
\newcommand{\tildeG}{\widetilde{G}}
\newcommand{\tildeR}{\widetilde{R}}

\xyoption{arrow}
\xyoption{curve}
\newdir{ (}{{}*!/-5pt/@^{(}}

\begin{document}

\newboolean{xlabels} 
\newcommand{\xlabel}[1]{ 
                        \label{#1} 
                        \ifthenelse{\boolean{xlabels}} 
                                   {\marginpar[\hfill{\tiny #1}]{{\tiny #1}}} 
                                   {} 
                       } 
\setboolean{xlabels}{false} 

\title[Determinantal Barlow surfaces and phantom categories]{Determinantal Barlow surfaces and phantom categories}

\author[B\"ohning]{Christian B\"ohning$^1$}
\address{Christian B\"ohning, Fachbereich Mathematik der Universit\"at Hamburg\\
Bundesstra\ss e 55\\
20146 Hamburg, Germany}
\email{christian.boehning@math.uni-hamburg.de}

\author[Bothmer]{Hans-Christian Graf von Bothmer$^2$}
\address{Hans-Christian Graf von Bothmer, Fachbereich Mathematik der Universit\"at Hamburg\\
Bundesstra\ss e 55\\
20146 Hamburg, Germany}
\email{hans.christian.v.bothmer@uni-hamburg.de}

\author[Katzarkov]{Ludmil Katzarkov$^3$}
\address{Ludmil Katzarkov, Department of Mathematics, University of Miami\\
Coral Gables, FL, 33146, USA \\
and Fakult\"at f\"ur Mathematik , Universit\"at Wien,\\
1090 Wien, Austria}
\email{lkatzark@math.uci.edu}

\author[Sosna]{Pawel Sosna$^2$}
\address{Pawel Sosna, Fachbereich Mathematik der Universit\"at Hamburg\\
Bundesstra\ss e 55\\
20146 Hamburg, Germany}
\email{pawel.sosna@math.uni-hamburg.de}

\thanks{$^1$ Supported by Heisenberg-Stipendium BO 3699/1-1 of the DFG (German Research Foundation)}
\thanks{$^2$ Supported by the RTG 1670 of the  DFG (German Research Foundation)}
\thanks{$^3$ Supported by grants NSF DMS0600800, NSF FRG DMS-0652633, NSF FRG DMS-0854977, NSF DMS-0854977, NSF DMS-0901330, grants FWF P 24572-N25 and FWF P20778, and an ERC grant --- GEMIS}

\begin{abstract}
We prove that the bounded derived category of the surface $S$ constructed by Barlow admits a length 11 exceptional sequence consisting of (explicit) line bundles. Moreover, we show  that in a small neighbourhood of $S$ in the moduli space of determinantal Barlow surfaces, the generic surface has a semiorthogonal decomposition of its derived category into a length 11 exceptional sequence of line bundles and a category with trivial Grothendieck group and Hochschild homology, called a phantom category. This is done using a deformation argument and the fact that the derived endomorphism algebra of the sequence is constant. Applying Kuznetsov's results on heights of exceptional sequences, we also show that the sequence on $S$ itself is not full and its (left or right) orthogonal complement is also a phantom category.
\end{abstract}

\maketitle

\section{Introduction}\xlabel{sIntroduction}

A (geometric) phantom category is an admissible subcategory $\mathcal{A}$ of the bounded derived category of coherent sheaves $\mathrm{D}^b(X)$ on a smooth projective variety $X$ with Hochschild homology $\mathrm{HH}_\ast(\mathcal{A})=0$ and Grothendieck group $\mathrm{K}_0(\mathcal{A}) = 0$.  Recently Katzarkov et al., \cite[Conj.\ 4.1]{DKK} and \cite[Conj.\ 29]{CKP}, conjectured that the derived category of the Barlow surface of \cite{Barlow1} should contain a phantom; this conjecture was based on the seminal works \cite{DON}, \cite{KOT}, \cite{OKV}, who used moduli spaces of instantons to distinguish smooth structures on Barlow surfaces and Del Pezzo surfaces of degree one, as well as on the works \cite{WIT}, \cite{AV} and \cite{ASP}, who studied the behaviour of D-branes under phase transitions. Further evidence for the possible existence of phantoms was given in the article \cite{BBS12}, where an admissible subcategory with vanishing Hochschild homology but with nonzero torsion Grothendieck group was produced in the derived category of the classical Godeaux surface. Shortly afterwards such ``quasi-phantoms" were  also found on Burniat surfaces by Alexeev and Orlov in \cite{A-O12}.  As we learned from S.\ Galkin at the conference ``Birational Geometry and Derived Categories'' in Vienna in August 2012, the preprint \cite{GMS} gives some arguments that make the existence of quasi-phantoms also plausible on fake projective planes with $3$-divisible canonical class. However, due to the rather complicated construction of the latter surfaces, one cannot yet prove this completely.

In this article we prove the existence of a phantom on a generic determinantal Barlow surface $S_t$ in a small neighbourhood of $S=S_0$ (the moduli space of determinantal Barlow surfaces is $2$-dimensional, see \cite{Cat81} or \cite{Lee00}), as well as on the Barlow surface $S$ itself. We think of $t$ as a deformation parameter. More precisely, our main result is the following.

\begin{theorem}\xlabel{tMain}
The derived category $\mathrm{D}^b (S_t)$ of a generic determinantal Barlow surface $S_t$ in a small neighbourhood of $S=S_0$ admits a semiorthogonal decomposition
\begin{gather*}
\mathrm{D}^b (S_t) = \langle \mathcal{A}_t, \mathcal{L}_{1, t}, \dots , \mathcal{L}_{11, t} \rangle
\end{gather*}
where $(\mathcal{L}_{1, t}, \dots ,\mathcal{L}_{11, t})$ is an exceptional sequence of line bundles and $\mathcal{A}_t$ is a phantom category. Moreover, if $S_{t_1}$ and $S_{t_2}$ are two surfaces in a small neighbourhood of a generic point of this family, then the categories $\langle \mathcal{L}_{1, t_1}, \dots ,\mathcal{L}_{11, t_1} \rangle$ and $\langle \mathcal{L}_{1, t_2}, \dots ,\mathcal{L}_{11, t_2} \rangle$ are equivalent. Furthermore, $\mathrm{D}^b(S)$ itself has a phantom.
\end{theorem}

After the discovery of the main results of this paper, we learned that Gorchinskiy and Orlov (\cite{GorOrl}) very recently produced another phantom category in the bounded derived category of a product of two surfaces by an ingenious and totally different method. 

Note that by \cite{Kawamata02} two minimal surfaces of general type with equivalent derived categories are isomorphic, so $\mathrm{D}^b (S_t)$ has to vary with the moduli of $S_t$. The way the moduli are encoded is analogous to what happens for Burniat surfaces in \cite{A-O12} (which was very inspiring for our proof). We prove that the $A_{\infty}$-Yoneda algebra of the exceptional sequence does not vary in a neighbourhood of a generic determinantal Barlow, and deduce from this the existence of the phantoms. Note that $\mathrm{K}_0 (S_t) \simeq \ZZ^{11}$ is torsion free, so we cannot use the torsion to prove that our exceptional sequence is not full. Likewise, we do not yet know how to exhibit explicit objects in $\mathcal{A}_t$ as was done in \cite{BBS12} (they all came from the fundamental group which is trivial here). It is a very interesting topic for future investigations to try to ``lay hands" on $\mathcal{A}_t$ and try to produce explicit objects in it or even explicitly describe a strong generator. 

Here is a short roadmap of the paper: in Section \ref{sConstruction} we recall the features of Barlow's construction of the surface $S$ which we will need later. In Section \ref{sLatticeTheory} we describe the symmetry of the classes of line bundles in the exceptional sequence we are going to construct. Section \ref{sCurvesBarlow} contains the construction of curves leading to an explicit integral basis in $\mathrm{Pic} (S)$, and the description of the intersection theory pertaining to it. In Section \ref{sExceptional} we explain how we obtain estimates for spaces of sections of line bundles on $S$ and prove the existence of the length $11$ exceptional sequence. In Section \ref{sDeformation}, we compute what we call cohomology data associated to this sequence, that is, the dimensions of extension groups (in the forward direction). Using a deformation argument, we prove existence of phantoms.  In Section \ref{sKuznetsov}, we prove that $S$ itself has a phantom using Kuznetsov's recent results on heights for exceptional sequences. In the last Section \ref{sConjectures}, some conjectures concerning possible applications of phantom categories are presented. 

We hope that the existence of phantom categories is not exclusively a pathology, but rather an interesting and useful structure in some derived categories of varieties. \smallskip

\textbf{Acknowledgments.} We would like to thank Denis Auroux, Fedor Bogomolov, Fabrizio Catanese, Igor Dolgachev, Sergey Galkin, Sergey Gorchinskiy, Sergei Gukov, Vladimir Guletskii, Fabian Haiden, Dmitry Kaledin, Alexander Kuznetsov, Maxim Kontsevich, Stefan M\"uller-Stach, Dmitri Orlov, Miles Reid, Eric Sharpe and Cumrun Vafa for useful discussions, helpful suggestions and remarks. Special thanks to Claire Voisin for making the results of \cite{Voi} available to us, thereby allowing us to render the treatment in Section \ref{sDeformation} complete.

\section{Notation and construction of the Barlow surface}\xlabel{sConstruction}

Let us recall the construction of determinantal Barlow surfaces in general. References are, for example, \cite{Barlow1}, \cite{Lee00} and \cite{Lee01}. Let $(x_1, \dots , x_4)$ be coordinates in $\PP^3$ and consider an action of $D_{10}= \langle \sigma, \tau \rangle$ on $\PP^3$ via
\begin{align*}
\sigma\colon (x_1, x_2, x_3, x_4) &\mapsto (\xi^1 x_1, \xi^2 x_2, \xi^3 x_3, \xi^4 x_4) , \\
\tau\colon (x_1, x_2, x_3, x_4) &\mapsto (x_4 , x_3, x_2, x_1)
\end{align*}
where $\xi$ is a primitive fifth root of unity. Then $D_{10}$-invariant symmetric determinantal quintic surfaces $Q$ in $\PP^3$ can be given as the determinants of the following matrices (see e.g. \cite{Lee00}, p. 898):
\[
A = \left( 
\begin{array}{ccccc}
0 & a_1x_1 & a_2 x_2 & a_2 x_3 & a_1 x_4 \\
a_1 x_1 & a_3x_2 & a_4x_3 & a_5 x_4 & 0\\
a_2 x_2 & a_4 x_3 & a_6 x_4 & 0 & a_5x_1 \\
a_2 x_3 & a_5 x_4 & 0 & a_6x_1 & a_4 x_2 \\
a_1 x_4 & 0 & a_5 x_1 & a_4 x_2 & a_3x_3
\end{array}
\right)
\]
where $a_1, \dots ,a_6$ are parameters. The generic surface $Q$ has an even set of $20$ nodes, so that there is double cover $\varphi_{K_Y} \colon Y \to Q$ with involution $\iota$ branched over the nodes. Here $\varphi_{K_Y}$ is the canonical morphism.  There is a twisted action of $D_{10}= \langle \sigma ,(\tau, \iota ) \rangle$ on $Y$ which has a group of automorphisms $H = \langle \sigma ,\tau \rangle \times \langle \iota \rangle = D_{10} \times \ZZ/2$.  Then  $X= Y/\langle \sigma ,(\tau, \iota ) \rangle$ is a surface with $4$ nodes whose resolution $\tilde{X}$ is a simply connected surface with $p_g=q=0$ (a determinantal Barlow surface) and $W = Y / \langle \sigma,\iota \rangle$ is a determinantal Godeaux surface (with $4$ nodes). This construction gives a $2$-dimensional moduli space of determinantal Barlow surfaces. The geometry is summarized in the following diagram

\[
\xymatrix{
    &   &   Y\ar[d]^{\ZZ/5 = \langle \sigma \rangle}\ar[lld]^{\ZZ/2 = \langle \iota \rangle }_{\varphi_{K_Y}}  \ar@/^6pc/[dd]_{D_{10}}^p  & & & \tilde{Y}\ar[lll]_{\tilde{\gamma}}\ar[dd]^{\tilde{p}} \\
Q \ar[d]^{\ZZ/5 = \langle \sigma \rangle }_{\pi } &    &   V   \ar[d]^{\ZZ/2 = \langle (\iota , \tau )\rangle}\ar[lld]^{\ZZ/2 = \langle \iota \rangle}  &        & & &         \\
W \ar[rd]_{\ZZ/2 = \langle \tau \rangle }  &   &  X \ar[ld]^{\ZZ/2 = \langle \iota \rangle }   & & &  \tilde{X}\ar[lll]^{\gamma}\\
    & \Sigma \ar@{.>}[d]^{\mathrm{bir.}}_{\simeq } &   &   & & &  \\
    & \PP^1 \times \PP^1 & & & & & 
}
\]
Here $V$ is a Campedelli surface, the double cover of the Godeaux surface $W$ ramified in the even set of four nodes of $W$. Thus $p_g(V) = q(V) =0$, $K^2_V =2$, $\pi_1 (V) = \ZZ/5$. 

The surface $Y$ has an explicit description as follows \cite{Cat81}, \cite{Reid81}: let
\[
R = \CC [x_1, \dots , x_4, y_0, \dots , y_4] /I
\]
where $\deg (x_i) =1$, $\deg (y_j) = 2$ and the ideal $I$ of relations is generated by 
\begin{gather*}
\sum_j A_{ij} y_j \quad\text{(5 relations in degree 3)} \\
 y_iy_j - B_{ij} \quad  \text{(15 relations in degree 4)} 
\end{gather*}
where $B_{ij}$ is the $(i,j)$-entry of the adjoint matrix of $A$ (in particular, $I$ contains $\det A$).  Then $Y$  is the subvariety in weighted projective space
\[
Y = \mathrm{Proj} (R) \subset \PP (1^4, 2^5 ) .
\]
This is a smooth (\cite[Prop.\ 2.11]{Cat81}) surface of general type with $p_g=4$, $q=0$, $K^2 =10$. 

\begin{remark}\xlabel{rClassical}
The special Barlow surface considered in \cite{Barlow1} corresponds to the choice of parameters
\[
a_1=a_2=a_4=a_5=1, \quad a_3=a_6=-4 .
\]
This can be seen by applying the base change 
\begin{align*}
	{x}_{1}&=5({X}_{1}+{X}_{2}+{X}_{3}+{X}_{4})\\
	{x}_{2}&=5(\xi {X}_{1}+\xi^{2} {X}_{2}+\xi^{3} {X}_{3}+\xi^4 {X}_{4})\\
	{x}_{3}&=5(\xi^2 {X}_{1}+(\xi^{2})^2 {X}_{2}+(\xi^{3})^2 {X}_{3}+(\xi^4)^2 {X}_{4})\\
	{x}_{4}&=5(\xi^3 {X}_{1}+(\xi^{2})^3 {X}_{2}+(\xi^{3})^3 {X}_{3}+(\xi^4)^3 {X}_{4})\\
	{x}_{5}&=5(\xi^4 {X}_{1}+(\xi^{2})^4 {X}_{2}+(\xi^{3})^4 {X}_{3}+(\xi^4)^4 {X}_{4})\\
	{y}_{0}&=\frac{1}{5}\left(\frac{Y_0}{6} + \xi^{2} {Y}_{1}+\xi^4 {Y}_{2}+ \xi {Y}_{3}+ \xi^{3} {Y}_{4}\right)\\
	{y}_{1}&=\frac{1}{5}\left(\frac{Y_0}{6} + \xi {Y}_{1}+ \xi^{2} {Y}_{2}+ \xi^{3} {Y}_{3}+\xi^4 {Y}_{4}\right)\\
	{y}_{2}&=\frac{1}{5}\left(\frac{Y_0}{6} + {Y}_{1}+ {Y}_{2}+ {Y}_{3}+ {Y}_{4}\right)\\
	{y}_{3}&=\frac{1}{5}\left(\frac{Y_0}{6} + \xi^4{Y}_{1}+ \xi^{3} {Y}_{2}+ \xi^{2} {Y}_{3}+ \xi {Y}_{4}\right)\\
	{y}_{4}&=\frac{1}{5}\left(\frac{Y_0}{6} + \xi^{3} {Y}_{1}+ \xi {Y}_{2}+\xi^4 {Y}_{3}+ \xi^{2} {Y}_{4}\right)
\end{align*}
to the setup given in \cite{Reid81}. We denote this special surface by $S$. It is distinguished by the fact that $Q$ is even invariant under a larger group $\mathfrak{S}_5$. 
\end{remark}

\begin{remark}\xlabel{rInvariants}
Invariants of $S$ are:
\begin{gather*}
K^2_S=1, \; p_g=q=0, \; \pi_1 (S) = \{ 1 \} , \\
K_0 (S) \simeq \ZZ^{11} , \; \mathrm{Pic} (S) \simeq H^2 (S, \ZZ )\simeq H_2 (S, \ZZ ) \simeq \ZZ^9 .
\end{gather*}

All integral cohomology classes on $S$ are algebraic. 

The least obvious statement that $\mathrm{K}_0 (S ) \simeq \mathbb{Z}^{11}$ follows from the fact that $\mathrm{Pic} (X) \simeq \ZZ^9$ and from the Bloch conjecture for $S$: $\mathrm{CH}^2 (S) \simeq \ZZ$ (this is known from \cite{Barlow2}). The argument is as follows: for surfaces we have 
\[\mathrm{rank}\colon F^0\mathrm{K} (S)/F^1\mathrm{K} (S)\cong \mathrm{CH}^0(S)\cong\ZZ,\]
\[c_1\colon F^1\mathrm{K} (W)/F^2\mathrm{K} (W)\cong \mathrm{Pic}(W),\]
\[c_2\colon F^2\mathrm{K}(S)\cong \text{CH}^2(S),\]
where $F^i \mathrm{K} (S)$ is the filtration of $\mathrm{K}_0 (S)$ by codimension of support. Moreover, $\mathrm{CH}^2 (S)$ is generated by the structure sheaf $\mathcal{O}_p$ of a point in $S$ and this is primitive in $\mathrm{K}_0 (S)$ (e.g. because $\chi (\mathcal{O}_p, \mathcal{O}_S ) = 1$). Then, looking at the sequence of extensions given by the filtration steps, one sees that $\mathrm{K}_0 (S) \simeq \ZZ^{11}$. 
\end{remark}

\begin{remark}\xlabel{rBasicRelations}
The following are some basic facts in this set-up.
\begin{itemize}
\item[(1)]
We have that $X$ and $W$ have rational singularities, $K_X$ and $K_W$ are invertible, and if $\pi\colon Q \to W$ is the projection, $(\pi\circ \varphi_{K})^{\ast } (K_W) = K_Y$. Moreover, $p^* K_X = K_Y$ and $\gamma^{\ast } (K_X ) = K_{\tildeX}$.
\item[(2)]
Locally around the four fixed points of the group $\ZZ/2 = \langle (\iota , \tau ) \rangle $, the quotient map $V \to X = V / (\ZZ/2 )$ looks like 
$\mathbb{A}^2 \to \mathrm{cone}\subset \mathbb{A}^3$ given by $(x,y) \mapsto (x^2, y^2, xy)$. 
\item[(3)]
The bundle $K_Y$ carries a canonical $D_{10}$-linearization corresponding to the $D_{10}$-action on $H^0 (Y, K_Y) \simeq \langle x_1, \dots ,x_4 \rangle$ given by the cycles $\sigma$ and $\tau$ as above.  In general, the action on $\bigoplus_{m\ge 0} H^0 (Y, mK_Y )$ is the one described in Remark \ref{rOperation} on $R$: this is the canonical ring. 
\end{itemize}
\end{remark}

\section{Lattice theory and semiorthonormal bases}\xlabel{sLatticeTheory}
We have $\mathrm{Pic} (S) = \mathbf{1} \perp (-E_8)$ as a lattice. We recall some facts from \cite{BBS12} which we will use in the sequel.

\begin{definition}\xlabel{dNumExceptional}
A sequence of classes $l_1, \dots , l_N$ in $\mathrm{K}_0 (S)$ is called \emph{numerically exceptional} if $\chi (l_i, l_i) =1$, for all $i$, and $\chi (l_i, l_j) = 0$ for $i > j$. 
\end{definition}

Let $A_1, \dots , A_8$ and $B_1, B_2$ be roots in $\mathrm{Pic} (S)$ with the following intersection behaviour:

\vspace{1cm}
\begin{center}
\setlength{\unitlength}{1cm}
\begin{picture}(8,1)
\put(0,0){$\bullet$}
\put(1,0){$\bullet$}
\put(2,0){$\bullet$}
\put(3,0){$\bullet$}
\put(4,0){$\bullet$}
\put(5,0){$\bullet$}
\put(6,0){$\bullet$}
\put(7,0){$\bullet$}

\put(2,1){$\bullet$}
\put(5,1){$\bullet$}

\put(0.1,0.1){\line(1,0){1}}
\put(1.1,0.1){\line(1,0){1}}
\put(2.1,0.1){\line(1,0){1}}
\put(3.1,0.1){\line(1,0){1}}
\put(4.1,0.1){\line(1,0){1}}
\put(5.1,0.1){\line(1,0){1}}
\put(6.1,0.1){\line(1,0){1}}

\put(2.1,0){\line(0,1){1}}
\put(5.1,0){\line(0,1){1}}

\put(0,-0.5){$A_1$}
\put(1,-0.5){$A_2$}
\put(2,-0.5){$A_3$}
\put(3,-0.5){$A_4$}
\put(4,-0.5){$A_5$}
\put(5,-0.5){$A_6$}
\put(6,-0.5){$A_7$}
\put(7,-0.5){$A_8$}

\put(1.5, 1){$B_1$}
\put(5.2,1){$B_2$}

\dashline{0.2}(2.1,1.1)(5.1,1.1)
\put(3.4,1.3){$-1$}
\end{picture}
\end{center}
\vspace{1cm}

Here, if two nodes are joined by a solid line, the intersection is $1$, otherwise it is zero. Moreover, $B_1$ and $B_2$ have intersection $-1$. 

\begin{proposition}\xlabel{pNumerics}
The sequence
\begin{align*}
A_1,\\
A_1+A_2,\\
k - B_1,\\
A_1+A_2+A_3, \\
A_1+A_2+A_3+A_4, \\
A_1+A_2+A_3+A_4+A_5, \\
k-B_2,\\
A_1+A_2+A_3+A_4+A_5+A_6, \\
A_1+A_2+A_3+A_4+A_5+A_6+A_7, \\
A_1+A_2+A_3+A_4+A_5+A_6+A_7+A_8, \\
\mathcal{O}
\end{align*}
is numerically exceptional of length $11$. 
\end{proposition}

This is \cite[Prop.\ 5.6]{BBS12}. The exceptional sequence we will construct on $S$ has this numerical behaviour. One advantage of this particular sequence is that the degrees of the differences of two classes in it are quite small, so there is a good chance to realize it as an actual exceptional sequence on $S$. The surface $S$ is homeomorphic to $\PP^2$ blown up in $8$ points, a del Pezzo surface of degree $1$, and the numerics of full exceptional sequences on del Pezzo surfaces has been thoroughly investigated (see, for example, \cite{KarNog}); also in this light, the sequence above seems to be most advantageous for our purposes.

\section{Curves on the Barlow surface and an explicit basis of the Picard group}\xlabel{sCurvesBarlow}

In this section we construct curves on the Barlow surface $S$. They will be used to make the intersection theory on the Barlow surface explicit. We will also use them to write down the exceptional sequence and to calculate sections of line bundles in section \ref{sExceptional}. 
In a first step we construct $D_{10}$-invariant curves on $Q$, pull them back to $Y$ and consider their images on $X$ and strict transforms on $\tildeX$. These curves are of degree $1$ and generate a $\textbf{1} \oplus (-D_8)$-sublattice of $\mathrm{Pic} (\tildeX )$. In a second step, using lattice theory, we find an effective divisor in the $\textbf{1} \oplus (-D_8)$-lattice which is divisible by $2$ as an effective divisor. The resulting divisor is of degree $2$. The degree $1$ curves together with this degree $2$ curve generate $\mathrm{Pic} (\tildeX )$ as a lattice. In a third step we use linkage and the automorphisms of $Y$ to construct $32$ curves of the same type as the degree $2$ curve above. 
Finally we calculate intersection numbers and write down our exceptional sequence and prove that the classes of the line bundles form a semiorthonormal basis of $\mathrm{K}_0 (\tildeX )$. 

The Macaulay2 scripts used to do the necessary calculations of this section and the following ones can be found at \cite{BBKS12}.

\begin{remark}\xlabel{rOperation}
The $D_{10}$-action on $Y$ (resp. the ambient $\PP (1^4, 2^5)$) is given by
 \begin{center}
\begin{tabular}{lcl}
	$\sigma(x_i)= \xi^i x_i$ & and & $\sigma(y_i)= \xi^{-i} y_i$ \\
	$\tau(x_i)=  x_{-i}$ & and &$\tau(y_i)= y_{-i}$ \\
	$\iota(x_i)= x_i$ & and  &$\iota(y_i)= -y_{i}$ \\
	$\alpha(x_i)= x_{\alpha(i)}$ & and & $ \alpha(y_i)= y_{\alpha(i)}$
\end{tabular}
\end{center}

with $\alpha = (1342)$. Moreover, we set $\beta = \iota\circ  \tau$. Then $D_{10} = \langle \sigma, \beta \rangle$ operates on $Y$.  The indices are interpreted as elements of $\ZZ/5$. The projection $Y\to Q$ is $D_{10}$-equivariant, where $\beta$ acts as $\tau$ on $Q$. Moreover, $\alpha$, $\tau$ and $\iota$ normalize the subgroup $D_{10}$. Hence they induce automorphisms on the quotient $X$ and we have $\tau = \iota = \alpha^2$.
\end{remark}

\newcommand{\goldenSection}{\Phi}

\newcommand{\EQ}{\overline{E}}
\newcommand{\LQ}{\overline{L}}
\newcommand{\PQ}{\overline{P}}

\begin{proposition}
The determinantal quintic $Q$ contains $15$ lines:
\begin{align*}
	\LQ^0_i &= \sigma^i(-t:-s:s:t )\\
	\LQ^+_i &= \sigma^i(s-\goldenSection t : -s+\goldenSection^{-1}t :-s :s+t )\\
	\LQ^-_i &= \sigma^i(s+\goldenSection^{-1} t : -s-\goldenSection t :-s :s+t )
\end{align*}
with $(s:t) \in \PP^1$ and $\goldenSection = (\sqrt{5} -1)/2$ is the golden section.
\end{proposition}

\begin{proof}
$\LQ^0_0$ is the unique $\tau$-invariant line on $Q$. A direct calculation also shows $\LQ^\pm_0 \subset Q$. The remaining lines
lie on $Q$ since $Q$ is $\sigma$-invariant. A direct calculation on the Grassmannian shows that there are at most $15$ lines on $Q$.  
\end{proof}

\begin{lemma}
The lines of the $\sigma$-orbit $\langle \LQ^0_0 \rangle$ are disjoint. The lines in the $\sigma$-orbits $\langle \LQ^\pm_0 \rangle$ form two pentagons.
\end{lemma}

\begin{proof}
Calculation.
\end{proof}

\newcommand{\Fp}{\FF_{421}}

\begin{proposition}
The $\tau$-invariant line $(s:t) \mapsto (t:s:s:t)$ intersects $Q$ in $5$ points. Over $\Fp$ the coordinates of these
points are
\begin{align*}
     \PQ_1 &= (-33:  1:   1: -33), \\ 
     \PQ_2 &= (     1: -33: -33:   1),\\
     \PQ_3 &= (-50:   1:   1: -50), \\
     \PQ_4 &= (1: -50: -50:   1), \\
     \PQ_5 &= (1:  -1:  -1:   1). 
\end{align*}
In particular, we have $\PQ_5 \in \LQ^\pm_0$.
\end{proposition}

\begin{proof}
Direct calculation.
\end{proof}

We recall the classification result for $\ZZ/5$-invariant elliptic quintics in $\PP^3$ due to Reid.

\begin{theorem}[\cite{Reid91}] \label{t:ellipticReid}
Let $E \subset \mathbb{P}^3$ be a $\ZZ/5$-invariant elliptic quintic curve not containing any coordinate points. Then
\begin{itemize}
\item the homogeneous ideal of E is generated by $5$ cubics of the form
\begin{align*}
R_0 &= a {x}_{1}^{2} {x}_{3}-b {x}_{1} {x}_{2}^{2}+c {x}_{3}^{2} {x}_{4}-d {x}_{2} {x}_{4}^{2} \\
R_1 &= a s {x}_{1} {x}_{2} {x}_{3}-a t {x}_{1}^{2} {x}_{4}-b s {x}_{2}^{3}-c t {x}_{3} {x}_{4}^{2}\\
R_2 &= a s {x}_{1} {x}_{3}^{2}-b s {x}_{2}^{2} {x}_{3}-b t {x}_{1} {x}_{2} {x}_{4}-d t {x}_{4}^{3}\\
R_3 &= a t {x}_{1}^{3}+c s {x}_{2} {x}_{3}^{2}+c t {x}_{1} {x}_{3} {x}_{4}-d s {x}_{2}^{2} {x}_{4}\\
R_4 &= b t {x}_{1}^{2} {x}_{2}+c s {x}_{3}^{3}-d s {x}_{2} {x}_{3} {x}_{4}+d t {x}_{1} {x}_{4}^{2},
\end{align*}
where $a,b,c,d$ are nonzero constants and $(s:t) \in \PP^1$. For $E$ to be nonsingular, we must have $\frac{tbc}{sad} \not\in \left\{0,\infty,\frac{-11 \pm 5\sqrt{5}}{2} = \left( \frac{-1 \pm \sqrt{5}}{2}\right)^5 \right\}$. The set of all $E$ is parametrised $1$-to-$1$ by $(s:t) \in \PP^1$ and the ratio $(a:b:c:d) \in \PP^3$.
\item The vector space of $\ZZ/5$-invariant quintic forms vanishing on $E$ has a basis consisting of the 7 elements
\[
	x_1^2R_3, 
	x_2^2R_1,
	x_3^2R_4,
	x_4^2R_2,
	x_1x_4R_0,
	x_2x_3R_0,
	x_3x_4R_3.
\]
\end{itemize}
\end{theorem}

From this one gets

\begin{proposition}\xlabel{pEllipticCurves}
$Q$ contains exactly $8$ $D_{10}$-invariant elliptic quintic curves. Their coordinates over $\Fp$ in Reid's parameter space are
\begin{align*}
     e_1 &= (  -1:  33:  -33:   1,  1:  202),\\
     e_2 &= ( -33:   1:   -1:  33, 202:   -1),\\
     e_3^+ &= (  -1:  50:  -50:  1,    1:  133),\\
     e_3^- &= (  -1:  50:  -50:  1,    1: -108),\\
     e_4^+ &= ( -50:   1:   -1: 50,  133:   -1),\\
     e_4^- &= ( -50:   1:   -1: 50, -108:   -1),\\
     e_5^+ &= (  -1:   1:   -1:   1,   1:  126),\\
     e_5^- &= (  -1:   1:   -1:   1, 126:   -1).
\end{align*}
We denote by $\EQ_i^\pm$ the elliptic quintic curve corresponding to $e_i^\pm$. We have $\PQ_i \in \EQ_j^\pm$ if
and only if $i=j$. Furthermore $\EQ_5^\pm = \langle \LQ^\pm_0 \rangle$ are the two pentagons.
\end{proposition}

\begin{proof}
Using Reid's setup we calculate the ideal of points on $\PP^3 \times \PP^1$ 
parametrizing $\ZZ/5\ZZ$-invariant elliptic quintic curves in $Q$. It turns out that this ideal
has degree $10$ and two solution points appear with multiplicity $2$. Over $\Fp$ we
obtain the same degrees and check that the above points are in the solution set by substitution.
From the form of the solutions we see that the $\EQ^\pm_i$ are also $\tau$-invariant. 
\end{proof}

\begin{remark}
The points $e_1$ and $e_2$ appear with multiplicity $2$ on Reid's parameter space.
\end{remark}

\begin{remark}\xlabel{rFieldDefinition}
The elliptic curves constructed in Proposition \ref{pEllipticCurves} are reductions of elliptic curves in characteristic $0$ since a calculation over $\mathbb{Q}$ shows that the number of such curves over $\CC$ is also $8$. 
\end{remark}

\begin{proposition}
The preimages of $\PQ_1$ and $\PQ_2$ on $Y$ are representatives of the branch locus of $p$.
Their coordinates on $Y$ are
\begin{align*}
          P_1^+ &= ( -33:   1:   1: -33: 0: -181:   53:  -53:  181), \\ 
	  P_1^- &= ( -33:   1:   1: -33: 0:  181:  -53:   53: -181), \\ 
     	  P_2^+ &= (   1: -33: -33:   1: 0:   53:  181: -181:  -53), \\
    	  P_2^- &= (   1: -33: -33:   1: 0:  -53: -181:  181:   53).
\end{align*}
\end{proposition}

\begin{proof}
Computation.
\end{proof}

\begin{notation}
We now pull back the curves constructed so far to $Y$ and denote them by $E_i^\pm$ and $\langle L \rangle$. 
Since they are $D_{10}$-invariant, they descend
to $X$. We then denote their strict transforms in $\tildeX$ by $\tildeL$ and $\tildeE_i^\pm$. The nodes of $X$ are
at the images of $P_i^\pm$, $i=1,2$. We denote their preimages on $\tildeX$ by
$\tildeC^\pm_i$. The whole configuration of the elliptic curves and the $(-2)$-curves on $\tildeX$ is visualized in Figure \ref{fConfiguration}. 
\end{notation}

\begin{lemma}\xlabel{lIntersection}
Let $\tildeD_1$, $\tildeD_2$ be two irreducible effective divisors on the Barlow surface $\tildeX$. Let $I_j$ be the ideal of $D_j= \tilde{\gamma} (\tilde{p}^{\ast} (\tildeD_j))$ on $Y$.  Put $I = I_1 + I_2$. We distinguish several cases. 

\begin{enumerate}
\item
$V(I)$ is empty. Then $\tildeD_1.\tildeD_2 = 0$.
\item $V(I_i)$ are ramification points. Then $\tildeD_1.\tildeD_2 = (-2) \degree V(I) / 5.$
\item $V(I_1)$ are ramification points and $V(I_2)$ is a curve which is smooth in all ramification points, or vice versa. Then $\tildeD_1.\tildeD_2 = \degree V(I) / 5.$ 
\item $V(I_1)$ and $V(I_2)$ are curves which are smooth in all ramification points and $V(I)$ is finite. Let $I_r$ be the ideal of the ramification locus of $p$. Then 
\[
\tildeD_1.\tildeD_2 = \frac{\deg I - \deg (I + I_r)}{10} .
\]
\item We have \[ \tildeK . \tildeD_1 = \frac{ \deg (I_1 + (x_1))}{10} . \]   
\item
$V(I_1)=V(I_2)=V(I)=D =D_1=D_2$ is a curve. Then \[
\tildeD^2 = \frac{2p_a (D) - 2 - \deg V(I + I_r) - \deg V(I) }{10}
\]
\end{enumerate}
\end{lemma}

\newcommand{\hatC}{\hat{C}}

\begin{proof}
Consider the diagram

\[
\xymatrix{
Y \ar[d]^p & \tilde{Y} \ar[l]_{\tilde{\gamma}} \ar[d]^{\tilde{p}}  & \hatC^{\pm}_{i,j} \ar@{_{(}->}
[l] \\
X & \tildeX \ar[l]_{\gamma} & \ar@{_{(}->}
[l] \tildeC^{\pm}_i
}
\]
Here $\tildeC^{\pm}_i$ are the four $(-2)$-curves on $\tildeX$ and $\hatC^{\pm}_{i,j}$, $j=1, \dots, 5$, are the twenty $(-1)$-curves lying over them. 
Assertions (1) and (2) are clear. 

In case (3) we have $\tildeD_1$ is one of the $(-2)$-curves, and $\tildeD_2$ is a curve intersecting all $\tildeC^{\pm}_i$ transversely. The number in (3) counts the intersection number of the $(-2)$-curve with $\tildeD_2$.

For (4) we compute  
\begin{align*}
\deg V(I) 
&=D_1.D_2 \\
&= \tilde{\gamma}^{\ast}(D_1).\tilde{\gamma}^{\ast}(D_2) \\
  &= (\tilde{p}^{\ast}(\tildeD_1) + \sum \delta_{ij}^\pm \hatC_{ij}^\pm ) . (\tilde{p}^{\ast}(\tildeD_2) + \sum \epsilon_{ij}^\pm \hatC_{ij}^\pm) \\
    &=  \tilde{p}^{\ast}(\tildeD_1) \tilde{p}^{\ast}(\tildeD_2) + 2\sum \delta_{ij}^\pm \cdot \epsilon_{ij}^\pm  -  \sum \delta_{ij}^\pm \cdot \epsilon_{ij}^\pm \\
    &=  \tilde{p}^{\ast}(\tildeD_1) \tilde{p}^{\ast}(\tildeD_2) + \sum \delta_{ij}^\pm \cdot \epsilon_{ij}^\pm   \\
    &= 10\tildeD_1.\tildeD_2   + \sum \delta_{ij}^\pm \cdot \epsilon_{ij}^\pm 
\end{align*}
where $\delta_{ij}^{\pm} = 1$ or $0$ depending on whether $D_1$ passes through $\tilde{\gamma} ( \hatC_{ij}^\pm )$ or not, and analogously for $\epsilon_{ij}^{\pm}$. This proves (4).

For (5) note that the formula is correct for $\tildeD_1$ a $(-2)$-curve because $x_1=0$ contains none of the ramification points of $p$. If $\tildeD_1$ is not a $(-2)$-curve, then, since $\tilde{\gamma}^* (K_Y) = \tilde{p}^* (\tildeK )$, 
\begin{align*}
\tildeK .\tildeD_1 &= \frac{1}{10}\tilde{p}^* (\tildeK ).\tilde{p}^*(\tildeD_1)\\ 
                               &= \frac{1}{10}\tilde{\gamma}^* (K_Y). \tilde{\gamma}^* (D_1)\\ 
                               &= \frac{1}{10}K_Y. D_1\\  
                               &= \frac{1}{10}\deg V ((x_1)+ I_1) .
\end{align*}
The second equality holds because $\tilde{p}^*(\tildeD_1) $ is equal to $\tilde{\gamma}^* (D_1)$ up to exceptional divisors on which $\tilde{\gamma}^* (K_Y) $ is trivial. 

In (6), $\tildeD_1= \tildeD_2 =: \tildeD$. The genus formula for $\tildeD$ yields
\[
\tildeD^2 = 2 p_a (\tildeD ) - 2 - \tildeK . \tildeD . 
\] 
The Hurwitz formula gives 
\[
2p_a (D) - 2= 2 p_a (\tilde{p}^* (\tildeD )) - 2 = 10 (2 p_a (\tildeD ) - 2 ) + \deg V(I + I_r)
\]
since $D$ is smooth in the ramification points of $p$. 
It follows
\[
\tildeD^2 = \frac{2p_a (D) - 2 - \deg V(I + I_r) - \deg V(I) }{10}
\]
\end{proof}

\begin{proposition} \xlabel{pIntersectionD8effective}
The intersection matrix  of the curves  
\[
	\bigl\{
	\tildeE_1,\tildeE_2,\tildeE_3^+,\tildeE_3^-,\tildeE_4^+,\tildeE_4^-,\tildeE_5^+,\tildeE_5^-,
	\tildeL,\tildeK,\tildeC_1^+,\tildeC_1^-,\tildeC_2^+\tildeC_2^-
	\bigr\},
\]
where $\tildeK$ is the
canonical divisor on $\tildeX$, is
\[
\makeatletter\c@MaxMatrixCols=14\makeatother
\begin{pmatrix}
{-1} &       0 &       0 &       0 &       0 &       0 &       0 &       0 &       {3} &       1 &       1 &       1 &       0 &       0\\
      0 &       {-1} &       0 &       0 &       0 &       0 &       0 &       0 &       {3} &       1 &       0 &       0 &       1 &       1\\
      0 &       0 &       {-1} &       1 &       0 &       0 &       0 &       0 &       {3} &       1 &       0 &       0 &       0 &       0\\
      0 &       0 &       1 &       {-1} &       0 &       0 &       0 &       0 &       {3} &       1 &       0 &       0 &       0 &       0\\
      0 &       0 &       0 &       0 &       {-1} &       1 &       0 &       0 &       {3} &       1 &       0 &       0 &       0 &       0\\
      0 &       0 &       0 &       0 &       1 &       {-1} &       0 &       0 &       {3} &       1 &       0 &       0 &       0 &       0\\
      0 &       0 &       0 &       0 &       0 &       0 &       {-1} &       1 &       {3} &       1 &       0 &       0 &       0 &       0\\
      0 &       0 &       0 &       0 &       0 &       0 &       1 &       {-1} &       {3} &       1 &       0 &       0 &       0 &       0\\
      {3} &       {3} &       {3} &       {3} &       {3} &       {3} &       {3} &       {3} &       {-3} &       1 &       0 &       0 &       0 &       0\\
      1 &       1 &       1 &       1 &       1 &       1 &       1 &       1 &       1 &       1 &       0 &       0 &       0 &       0\\
      1 &       0 &       0 &       0 &       0 &       0 &       0 &       0 &       0 &       0 &       {-2} &       0 &       0 &       0\\
      1 &       0 &       0 &       0 &       0 &       0 &       0 &       0 &       0 &       0 &       0 &       {-2} &       0 &       0\\
      0 &       1 &       0 &       0 &       0 &       0 &       0 &       0 &       0 &       0 &       0 &       0 &       {-2} &       0\\
      0 &       1 &       0 &       0 &       0 &       0 &       0 &       0 &       0 &       0 &       0 &       0 &       0 &       {-2}\\
      \end{pmatrix}
\]
The rank of this matrix is $9$. See also Figure \ref{fConfiguration}. 
\end{proposition}

\begin{proof}
We calculate the intersection numbers on $Y$. For this we first check that $E_i^\pm$, $\langle L \rangle$ are smooth in
the $P^\pm_i$. By the $D_{10}$-invariance of the orbits this shows that they are smooth in all branch points of $p$. We also represent
$K$ by the curve $\{x_1=0\}$ and set $C^\pm_i = \langle P^\pm_i \rangle$. The assertion follows by Lemma \ref{lIntersection}.
\end{proof}

\begin{figure}
\begin{center}
\includegraphics[height=90mm]{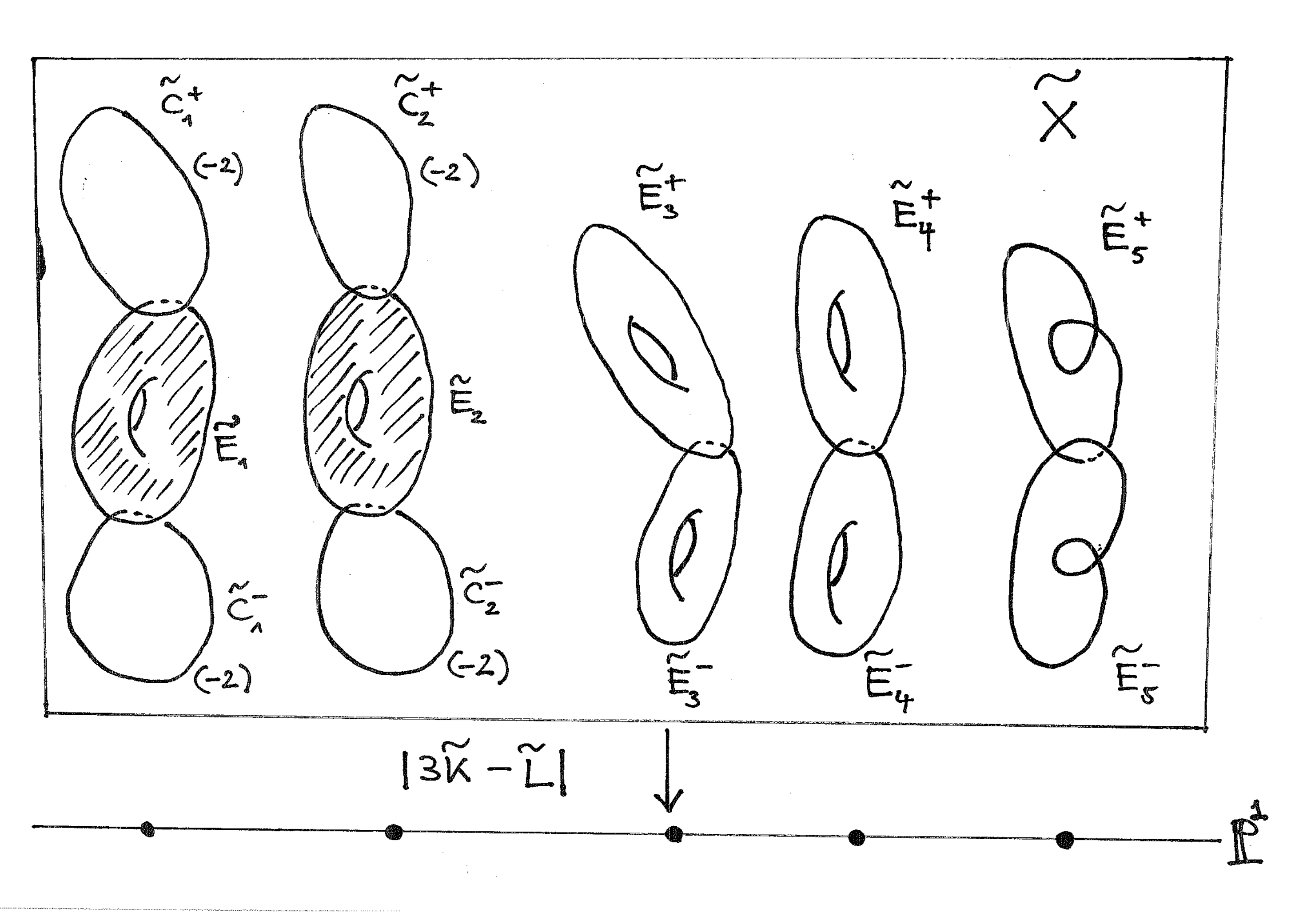}
\end{center}
\caption{The configuration of curves on the Barlow surface} \label{fConfiguration} 
\end{figure}

\begin{remark}\xlabel{rFiniteFieldIntersection}
The intersections are calculated over a finite field and may potentially be different from those in characteristic zero. The way the argument works is however the following: we produce eventually an exceptional sequence
\[
(\overline{\mathcal{L}}_1, \dots , \overline{\mathcal{L}}_{11})
\]
of line bundles on the reduction of the Barlow surface to finite characteristic. However, the $\overline{\mathcal{L}}_i$ themselves are reductions of line bundles $\mathcal{L}_i$ defined over an algebraic number field of characteristic $0$. Hence by upper-semicontinuity over $\mathrm{Spec}(\mathfrak{O})$, where $\mathfrak{O}$ is the ring of integers of this number field, the sequence
\[
( \mathcal{L}_1, \dots , \mathcal{L}_{11})
\]
will also be exceptional in characteristic $0$.
\end{remark}

\begin{remark}
We have $14$ effective possibly reducible elliptic curves of degree $1$ on $\tildeX$. These are
\[
	\{\tildeE_i,\tildeE_i+\tildeC_i^+,\tildeE_i+\tildeC_i^-,\tildeE_i+\tildeC_i^++\tildeC_i^-\}
\]
for $i=1,2$ and $\tildeE_j^\pm$ for $j=3,4,5$.
\end{remark}

\begin{proposition}
On $\tildeX$ we have the following roots (so far)
\begin{itemize}
	\item $\tildeF_i - \tildeF_j$ with $\tildeF_i.\tildeF_j = 0$ and $\tildeF_i,\tildeF_j$ possibly reducible elliptic curves of degree one.
	         (84 of these)
	\item $\pm \tildeK \mp \tildeF_i$ with $\tildeF_i$ as above (28 of these).
\end{itemize}
These $112$ roots form a $D_8$-root system. A $D_8$-basis is given, for example, by the simple roots
\begin{align*}
	\sD_1 &=  \tildeK-\tildeE_2-\tildeC_2^- \\
    	\sD_2 &= \tildeE_2+\tildeC_2^--\tildeE_2 \\
     	\sD_3 &= \tildeE_1 -\tildeE_2-\tildeC_2^+-\tildeC_2^-  \\
  	\sD_4 &= \tildeE_1+\tildeC_1^++\tildeC_1^- -\tildeE_3^- \\
   	\sD_5 &= \tildeE_3^--\tildeE_4^-\\
  	\sD_6 &= \tildeE_4^--\tildeE_5^-\\
  	\sD_7 &= \tildeE_1+\tildeC_1^--\tildeE_5^+\\
  	\sD_8 &= \tildeE_1+\tildeC_1^+-\tildeE_5^+
 \end{align*}

They have intersection matrix 
\[
\begin{pmatrix}{-2} &         1 &         0 &         0 &         0 &         0 &         0 &         0\\
        1 &         {-2} &         1 &         0 &         0 &         0 &         0 &         0\\
        0 &         1 &         {-2} &         1 &         0 &         0 &         0 &         0\\
        0 &         0 &         1 &         {-2} &         1 &         0 &         0 &         0\\
        0 &         0 &         0 &         1 &         {-2} &         1 &         0 &         0\\
        0 &         0 &         0 &         0 &         1 &         {-2} &         1 &         1\\
        0 &         0 &         0 &         0 &         0 &         1 &         {-2} &         0\\
        0 &         0 &         0 &         0 &         0 &         1 &         0 &         {-2}\\
        \end{pmatrix}
\]
\end{proposition}

\begin{proof}
Calculation using the above intersection matrix in Proposition \ref{pIntersectionD8effective}.
\end{proof}

\begin{remark}
All effective curves constructed so far can be written in this $D_8$-basis and $\tildeK$ using integer coefficients. 
\end{remark}

\begin{proposition}
On $\tildeX$ there exist $32$ curves $\tildeB^\pm_{ijk}$, $i,j \in \ZZ/2\ZZ$ and $k \in \{0,1,2,3\}$, of genus $3$ and canonical degree $2$ each intersecting two $(-2)$-curves, say
$\tildeF^\pm_{ijk}$ and $\tildeG^\pm_{ijk}$, such that 
\begin{align*}
	&2\tildeK - \tildeB^\pm_{ijk}, \\
	&2\tildeK - \tildeB^\pm_{ijk} - \tildeF^\pm_{ijk}, \\
	&2\tildeK - \tildeB^\pm_{ijk} - \tildeG^\pm_{ijk}, \\ 
	&2\tildeK - \tildeB^\pm_{ijk} - \tildeF^\pm_{ijk} - \tildeG^\pm_{ijk}
\end{align*}
represent $128$ additional roots. The total of $112+128=240$ roots forms an $E_8$-lattice.
\end{proposition}

\begin{proof}
Let $d_1,\dots,d_8$ be a system of simple roots in a $D_8$-lattice. If this lattice is a sublattice of an $E_8$-lattice, the
Borel-Siebenthal algorithm (see e.g. \cite[13.2, p.\ 109]{MT}) gives the highest root in the $E_8$-lattice as
\[
	e = \frac{1}{2}(d_1 + 2d_2 +3d_3 +4d_4 + 5d_5 + 6d_6 + 3d_7 + 4d_8)
\]
and $\{e,d_8,\dots,d_2\}$ is a system of simple roots in the $E_8$-lattice.

Now observe that 
$$\chi(e+2\tildeK) = \frac{(e+2\tildeK).(e+\tildeK) + 2}{2} = \frac{-2+2+2}{2} = 1.$$ 
Since
$H^2(e+2\tildeK) = H^0(-\tildeK-e) = 0$, this implies that $H^0(e+2\tildeK) \not=0$. We will construct a curve
$B \in |e+2\tildeK|$. By construction $B$ is not in the subgroup of $\Pic \tildeX$ considered so far, but
$2B$ is, i.e. there exist a non-reduced curve in the linear system $|2B|$ whose support is $B$.
To represent $|2B|$ in a computer, we write
\[
	2B \equiv 8\tildeK-\tildeL-\tildeE_3^+ -\tildeE_4^+ -\tildeE_5^+ + \tildeC_1^+ -\tildeC_2^+
\]
and, therefore, consider $D_{10}$-invariant polynomials of degree $8$ that lie in the Ideal $I$
of $\langle L \rangle \cup E_3^+ \cup E_4^+  \cup E_5^+  \cup \langle P_1^+ \rangle$ in $Y$. A computation shows that
there is a $\PP^3$ of such polynomials. By restricting to lines we find the unique such polynomial $F$ that
is non-reduced on a curve outside of $L \cup E_3^+ \cup E_4^+  \cup E_5^+$. The ideal of the curve $B$
is then obtained as $\rad( ((F)+I(Y)) : I )$. It is $D_{10}$-invariant and hence descends to $X$. We denote its strict transform
by $\tildeB_{000}^+$. 

The intersection of $\tildeB_{000}^+$ with the effective curves of Proposition \ref{pIntersectionD8effective} can be calculated to be
\[
	\{2, 2, 3, 2, 3, 2, 3, 2, 1, 2, 0, 1, 1, 0 \} 
\]

Now let $\tildeE$ be an elliptic curve and $\tildeB$ be a genus $3$ curve of degree $2$ with $\tildeE.\tildeB = 3$ on $\tildeX$. 
Then we have
\[
	\chi(5\tildeK - \tildeE-\tildeB) = 1
\]
We thus have an effective curve $\tildeB'  \in |5\tildeK - \tildeE-\tildeB|$, and since
$\tildeB^2=2$, we have $g(\tildeB') = 3$ and $\tildeB'.\tildeK = 2$. We say that $\tildeB'$ is linked
via $5\tildeK$ to $\tildeB + \tildeE$. Performing this construction with $\tildeB^+_0$ and 
$\tildeE_3^+$,  $\tildeE_4^+$ and $\tildeE_5^+$, we obtain further genus $2$ curves $\tildeB^+_{001},\tildeB^+_{002},\tildeB^+_{003}$.
Observe that, furthermore,
\[
	\chi(4\tildeK - \tildeB) = 1
\]
and that for $\tildeB' \in |4\tildeK-\tildeB|$ we also have $g(\tildeB') = 3$ and $\tildeB'.\tildeK = 2$. 
Therefore we can link $\tildeB^+_{00k}$ via $4\tildeK$ to $\tildeB^-_{00k}$. Now we set
\[
	\tildeB^\pm_{ijk} = \iota^i(\alpha^j(\tildeB^\pm_{00k}))
\]
and obtain a total of $32$ curves. With a computer we can check that all of the
$32$ constructed curves are distinct and that each of them intersects exactly two $(-2)$ curves. 
\end{proof}

\begin{proposition}
We have the following intersections on $\tildeX$:
\[
	(\tildeB_{ijk}^\pm - 2\tildeK) . \tildeL = \mp1
\]
\end{proposition}

\begin{proof}
Calculation.
\end{proof}

\begin{proposition}
The exceptional curves $\tildeC_i^\pm$ have the following intersections with $\tildeB_{ijk}^\pm-2\tildeK$ on $\tildeX$:
\begin{center}
\begin{tabular}{|c|c|c|c|c|}
	\hline
	 & \multicolumn{2}{|c|}{$i=0$} & \multicolumn{2}{|c|}{$i=1$} \\ 
	 \cline{2-5}
	 & $j=0$ & $j=1$ & $j=0$ & $j=1$ \\ 
	 \hline
	$\tildeC_1^+$ & $0$ & $1$ & $0$ & $1$ \\
	$\tildeC_1^- $ & $1$ & $0$ & $1$ & $0$ \\	
	$\tildeC_2^+$ & $1$ & $0$ & $0$ & $1$ \\
	$\tildeC_2^-$  & $0$ & $1$ & $1$ & $0$ \\
	\hline
\end{tabular}	
\end{center}
\end{proposition}

\begin{proof}
Calculation.
\end{proof}

\begin{proposition}
We have the following intersections with the elliptic curves $\tildeE_1$ and $\tildeE_2$:
\begin{align*}
	(\tildeB_{ijk}^+ - 2\tildeK) . \tildeE_i &= 1 \\
	(\tildeB_{ijk}^- - 2\tildeK) . \tildeE_i &=  0
\end{align*}
For the elliptic curves $\tildeE_3^\pm$ and $\tildeE_4^\pm$ we have:
\begin{center}
\begin{tabular}{|c|c||c|c|c|c|}
	\hline
	 && $k=0$ & $k=1$ & $k=2$ & $k=3$ \\ 
	 \hline
	\multirow{4}{*}{$\tildeB_{0jk}^\pm - 2\tildeK$} &  $\tildeE_3^+$ &  $\pm 1$ &  $\pm 1$ &  $0$ &  $0$\\
	&  $\tildeE_3^-$ &  $0$ &  $0$ &  $\pm 1$ &  $\pm 1$ \\
	&  $\tildeE_4^+$ &  $\pm 1$ &  $0$ &  $\pm 1$ &  $0$ \\
	&  $\tildeE_4^-$ &  $0$ &  $\pm 1$ &  $0$ &  $\pm 1$ \\
 	\hline
	\multirow{4}{*}{$\tildeB_{1jk}^\pm - 2\tildeK$} &  $\tildeE_3^+$ &  $0$ &  $\pm 1$ &  $0$ &  $\pm 1$ \\
	&  $\tildeE_3^-$ &  $\pm 1$ &  $0$ & $\pm 1$ &  $0$ \\
	&  $\tildeE_4^+$ &  $0$ &  $0$ &  $\pm 1$ & $\pm 1$ \\
	&  $\tildeE_4^-$ &  $\pm 1$ &  $\pm 1$ &  $0$ &  $0$ \\
	\hline
\end{tabular}	
\end{center}
For the elliptic curves  $\tildeE_5^\pm$ we have:
\begin{center}
\begin{tabular}{|c|c|c|c|c|c|}
	\hline
	 & $k=0$ & $k=1$ & $k=2$ & $k=3$ \\ 
	 \hline
	$(\tildeB^\pm_{0jk} -2\tildeK).\tildeE_5^-$ & $0$ & $\pm 1$ & $\pm 1$& $0$ \\
          $(\tildeB^\pm_{1jk} -2\tildeK).\tildeE_5^+$ & $0$ & $\pm 1$& $\pm 1$& $0$ \\
          $(\tildeB^\pm_{1jk} -2\tildeK).\tildeE_5^-$ & $\pm 1$ & $0$ & $0$ & $\pm 1$\\
           $(\tildeB^\pm_{0jk} -2\tildeK).\tildeE_5^+$ & $\pm 1$ & $0$ & $0$ & $\pm 1$ \\
	 \hline
\end{tabular}
\end{center}

\end{proposition}

\begin{proof}
Calculation.
\end{proof}

\begin{proposition}
Let $i,j,i',j' \in \{0,1\}$ and $q,q' \in \{+1,-1\}$. If $i=i'$, then
\[
	( \tildeB_{ijk}^q.\tildeB_{i'j'k'}^{q'}  )_{k=0,\dots, 3, k' = 0, \dots , 3} 
	= \begin{pmatrix}
		b & a & a & a \\
		a & b & a & a \\
		a & a & b & a \\
		a & a & a & b 
	    \end{pmatrix}
\]
with $a = 3+ ((j+j') \mod 2)$ and  $b=a-qq'$.  If $i\not=i'$ then
\[
	( \tildeB_{ijk}^q.\tildeB_{i'j'k'}^{q'}  )_{k=0,\dots, 3, k' = 0, \dots , 3} 
	= \begin{pmatrix}
		a & a & a & b \\
		a & b & a & a \\
		a & a & b & a \\
		b & a & a & a
	    \end{pmatrix}
\]
with $a = 4-(qq'+1)/2$ and  $b=3+(qq'+1)/2$.
\end{proposition}

\begin{proof}
Calculation.
\end{proof}

Using these roots, we can use the method of Section \ref{sLatticeTheory} to obtain explicitly a numerically semi-orthogonal sequence of 
line bundles:

\begin{proposition}\xlabel{pSequence}
The following sequence of line bundles is numerically semi-orthogonal:
\begin{align*}
	\sL_1 &= \tildeE_1-\tildeE_2, \\
	\sL_2 &= \tildeE_3^+ - \tildeE_2, \\
	\sL_3 &= 2\tildeK-\tildeE^+_4, \\
	\sL_4 &= \tildeE_4^- - \tildeE_2, \\
	\sL_5 &= 2\tildeK-\tildeB^-_{012}-\tildeC_1^+, \\
	\sL_6 &= 2\tildeK-\tildeB^-_{002}-\tildeC_1^-, \\
	\sL_7 &= 2\tildeK-\tildeE^+_5, \\
	\sL_8 &= 2\tildeK - \tildeB^-_{111} -\tildeC_1^+, \\
	\sL_9 &= 2\tildeK - \tildeB^-_{101} -\tildeC_1^-, \\
	\sL_{10} &= \sO, \\
	\sL_{11} &= \tildeK-\tildeE_2.
\end{align*}

\end{proposition}

\begin{proof}
The following matrix contains $\chi(\sL_i-\sL_j)$ at the $(i,j)$-th entry:
\[
\makeatletter\c@MaxMatrixCols=11\makeatother\begin{pmatrix}
        1 &         0 &         0 &         0 &         0 &         0 &         0 &         0 &         0 &         0 &         0\\
        0 &         1 &         0 &         0 &         0 &         0 &         0 &         0 &         0 &         0 &         0\\
        {-1} &         {-1} &         1 &         0 &         0 &         0 &         0 &         0 &         0 &         0 &         0\\
        0 &         0 &         1 &         1 &         0 &         0 &         0 &         0 &         0 &         0 &         0\\
        0 &         0 &         1 &         0 &         1 &         0 &         0 &         0 &         0 &         0 &         0\\
        0 &         0 &         1 &         0 &         0 &         1 &         0 &         0 &         0 &         0 &         0\\
        {-1} &         {-1} &         0 &         {-1} &         {-1} &         {-1} &         1 &         0 &         0 &         0 &         0\\
        0 &         0 &         1 &         0 &         0 &         0 &         1 &         1 &         0 &         0 &         0\\
        0 &         0 &         1 &         0 &         0 &         0 &         1 &         0 &         1 &         0 &         0\\
        0 &         0 &         1 &         0 &         0 &         0 &         1 &         0 &         0 &         1 &         0\\
        0 &         0 &         1 &         0 &         0 &         0 &         1 &         0 &         0 &         0 &         1\\
        \end{pmatrix}.
\]
\end{proof}

\section{Sections in line bundles and the exceptional sequence}\xlabel{sExceptional}

Here we explain how we calculate sections of line bundles on $S$, or rather, obtain upper bounds for the dimensions of the spaces of sections in these line bundles. 

Look at the natural commutative diagram of $G$-varieties (recall $G=D_{10}$ here)
\[
\xymatrix{\tilde{Y} \ar[d]_{\tilde{p}} \ar[r]^{\tilde{\gamma}} & Y \ar[d]^p \\
					S\ar[r]^\gamma & X.}
\]
We write 
\[
D \equiv nK_S - P
\]
where $P$ is an effective divisor, $n\in \mathbb{N}$. Note that numerical equivalence coincides with linear equivalence on $S$. We find $P$ using integer programming (\cite{BBKS12}). 

Write $P = P' + \sum a_i \tilde{C}_i$ where $\tilde{C}_i$ are the $(-2)$-curves on $S$ (these are the $\tilde{C}_k^{\pm}$ of Section \ref{sCurvesBarlow}; here the index $i$ runs from $1$ to $4$), and $P'$ does not contain $(-2)$-curves as components. 

Let $P''=\tilde{\gamma}(\tilde{p}^* (P'))$.  

\begin{proposition}\xlabel{pSectionEstimate}
We have 
\[
\dim \left( H^0 (Y,  \mathcal{O}_Y (nK_Y - P'' ))^{G} \right) \ge \dim H^0 (S, \mathcal{O}_S (D))\, .
\]
If $\mathfrak{p}$ is some prime number and we denote reduction by the index $\mathfrak{p}$ (all data are defined over $\ZZ$), the analogous inequality holds:
\[
\dim \left( H^0 (Y_{\mathfrak{p}} ,  \mathcal{O}_{Y_{\mathfrak{p}}} (nK_{Y_{\mathfrak{p}}}  - P''_{\mathfrak{p}} ))^{G} \right) \ge \dim H^0 (S, \mathcal{O}_S (D))\, .
\]
\end{proposition}

\begin{proof}
We have that 
\begin{align*}
\dim H^0 (S, \mathcal{O}_S (D)) &\le \dim H^0 (\tilde{Y}, n\tilde{p}^* (K_S) - \tilde{p}^* (P))^{G}\\
 & = \dim H^0 (\tilde{Y}, n\tilde{p}^* (K_S) - \tilde{p}^* (P') - \sum_{ij} 2 a_i \hat{C}_{ij} )^{G}\\
 & \le  \dim H^0 (\tilde{Y}, n\tilde{p}^* (K_S) - \tilde{p}^* (P'))^G
\end{align*}
since $\mathcal{O}_{\tilde{Y}}(n\tilde{p}^* (K_S) - \tilde{p}^* (P') - \sum_{ij} 2 a_i \hat{C}_{ij})$ is a $G$-equivariant subbundle of $\mathcal{O}_{\tilde{Y}}(n\tilde{p}^* (K_S) - \tilde{p}^* (P'))$. Now
\[
\tilde{p}^*(K_S) = \tilde{\gamma}^*(K_Y) = K_{\tilde{Y}} - \sum_{ij}\hat{C}_{ij}
\]
and $\tilde{p}^* P' =\tilde{\gamma}^* (P'') - \sum_{ij} b_{ij}\hat{C}_{ij}$, $b_{ij}\ge 0$, hence
\begin{align*}
\dim H^0 (\tilde{Y}, n\tilde{p}^* (K_S) - \tilde{p}^* (P'))^G &= \dim H^0 (\tilde{Y}, \tilde{\gamma}^*(nK_Y -P'') + \sum_{ij} b_{ij}\hat{C}_{ij})^G .
\end{align*}
By the projection formula, and since $\tilde{\gamma}$ is $G$-equivariant, we get
\[
\tilde{\gamma}_{\ast} \left(  \mathcal{O}_{\tilde{Y}} (\tilde{\gamma}^*(nK_Y -P'') + \sum_{ij} b_{ij}\hat{C}_{ij}) \right) = \mathcal{O}_{Y} (nK_Y -P'')
\]
as $G$-bundles. This proves the first inequality and the second part follows using upper-semicontinuity over $\mathrm{Spec} (\ZZ )$.
\end{proof}

We will use the second inequality in Proposition \ref{pSectionEstimate} to obtain bounds on dimensions of spaces of sections: we will put $\mathfrak{p} = 421$ and compute the dimension of the space of degree $n$ polynomials in the coordinates $x_1, \dots , x_4, y_0, \dots , y_4$ on the weighted projective space $\mathbb{P} (1^4, 2^5)$ vanishing on $p^* (P'_{\mathfrak{p}})$ modulo the space of degree $n$ polynomials in the homogeneous ideal of $Y$ in $\mathbb{P} (1^4, 2^5)$ with a computer algebra system (Macaulay2). Note that $\CC [x_1, \dots , x_4, y_1, \dots , y_5 ] / I$ is indeed exactly the canonical ring of $Y$, thus its elements give the pluricanonical sections on $Y$.

We obtain the following vanishing theorem:

\begin{proposition}
Let $\tildeR \in \Pic \tildeX$ be a root, i.e $\tildeR.\tildeK = 0$ and $\tildeR^2=-2$. Then we have
\begin{itemize}
\item $h^0(\tildeR) \not=0$ if and only if $\tildeR \cong C^\pm_i$. In this case $h^1(R) = 1$ and $h^2(R) = 0$.
\item $h^2(\tildeR) \not=0$ if and only if $\tildeK-\tildeR \cong \tildeE \in \{\tildeE_1,\tildeE_2,\tildeE_3^\pm,\dots,\tildeE_5^\pm\}$. In this case $h^1(\tildeR) = 1$ and $h^0(\tildeR) = 0$.
\end{itemize}
\end{proposition}

\begin{proof}
The if part is obvious. The reverse can be checked for all $240$ roots by calculating directly over $\Fp$ using Proposition \ref{pSectionEstimate}.
\end{proof}

We obtain furthermore

\begin{theorem}\xlabel{tSequence}
The sequence of line bundles given in Proposition \ref{pSequence} is exceptional on $S$: $\mathrm{RHom}^{\bullet} (\mathcal{L}_j, \mathcal{L}_i ) =0$ for $j > i$.
\end{theorem}

\section{The deformation argument and and existence of the phantoms}\xlabel{sDeformation}

In this section we will prove the existence of phantom categories in $\mathrm{D}^b (S_t)$ where $S_t$ is generic in the moduli space of determinantal Barlow surfaces in a small neighbourhood of the distinguished Barlow surface $S=S_0$ of Section \ref{sConstruction}. 

\begin{lemma}\xlabel{lDeformation}
Let $S_t$ be a generic determinantal Barlow surface in a small neighbourhood of $S$. Then there is an exceptional sequence $(\mathcal{L}_{1, t}, \dots , \mathcal{L}_{11, t})$ in $\mathrm{D}^b (S_t)$ consisting of line bundles $\mathcal{L}_{i, t}$ which are deformations of the $\mathcal{L}_i$.
\end{lemma}

\begin{proof}
Consider a small nontrivial deformation of $S$ (one  deformation parameter $t$ for simplicity) among determinantal Barlow surfaces: 
\[
\xymatrix{S \ar@{=}[r]& \mathcal{S}_0 \ar@{ (->}[r] \ar[d] & \mathcal{S}\ar[d]  \\
& 0 \ar@{ (->}[r] & B}
\]
The line bundles $\mathcal{L}_i$ deform to line bundles $\mathcal{L}_{i,t}$ and by upper semicontinuity, the $\mathcal{L}_{i,t}$ are also an exceptional sequence.
\end{proof}

Consider now $\mathbb{L}_t = \bigoplus_{i=1}^{11} \mathcal{L}_{i, t}$ and the differential graded algebra $\mathfrak{A}_t = \mathrm{RHom}^{\bullet} (\mathbb{L}_t, \mathbb{L}_t)$ of derived endomorphisms of the exceptional sequence \[ (\mathcal{L}_{1,t}, \dots , \mathcal{L}_{11,t})\] above. It has a minimal model in the sense of \cite[Sect.\ 3.3]{Keller01}, that is, we consider the Yoneda algebra $H^* (\mathfrak{A}_t)$ together with its $A_{\infty}$-structure such that $m_1=0$, $m_2=$Yoneda multiplication and there is a quasi-isomorphism of $A_{\infty}$-algebras $\mathfrak{A}_t \simeq H^* (\mathfrak{A}_t)$ lifting the identity of $H^* (\mathfrak{A}_t)$. We will show  

\begin{proposition}\xlabel{pRigidity}
In the neighbourhood of a generic value of the deformation parameter $t$, the algebra $H^*(\mathfrak{A}_t)$  is constant. Hence, the subcategories $\langle \mathcal{L}_{1,t}, \dots , \mathcal{L}_{11, t}  \rangle$ are all equivalent in a neighbourhood of a generic value of $t$.
\end{proposition} 

In fact, we are dealing here with an $A_{\infty}$-category on the $11$ objects $\mathcal{L}_{i,t}$. Let us recall now some facts about $A_{\infty}$-categories which we need to prove Proposition \ref{pRigidity}. A possible reference is the first chapter in Seidel's book \cite{Seidel}. In particular, in an $A_{\infty}$-category we are given a set of objects $X_i$ with a graded vector space $\mathrm{hom} (X_0, X_1)$ for any pair of objects, and composition maps of every order $d\ge 1$
\[
\mathrm{hom} (X_0, X_1) \otimes \mathrm{hom} (X_1, X_2) \otimes \dots \otimes \mathrm{hom}(X_{d-1}, X_d) \to \mathrm{hom} (X_0, X_d)[2-d] 
\]
satisfying the $A_{\infty}$-associativity equations, whose precise form we actually need not know here. The important point is that $m_d$ is homogeneous of degree $2-d$. Another important point is cf. \cite[Lem.\ 2.1]{Seidel} that any homotopy unital $A_{\infty}$-category is quasi-isomorphic to a strictly unital one, i.e. we may assume
\[
m_d (a_0\otimes \dots \otimes a_{i-1}\otimes \mathrm{id} \otimes a_{i+1} \otimes \dots \otimes a_d ) = 0, d\ge 3
\]
which means $m_d$, $d\ge 3$, is zero as soon as one of its arguments is a homothetic automorphism of an object.

We use

\begin{lemma}\xlabel{lRigidity}
The following matrix describes the $\mathrm{Ext}^n (\mathcal{L}_{i, t}, \mathcal{L}_{j, t} )$ arising from the exceptional sequence. More precisely, the $(i,j)$-entry of the matrix is $[\dim \mathrm{Hom} (\mathcal{L}_{i, t}, \mathcal{L}_{j, t}), \dim \mathrm{Ext}^1 (\mathcal{L}_{i, t}, \mathcal{L}_{j, t}), \dim \mathrm{Ext}^2(\mathcal{L}_{i, t}, \mathcal{L}_{j, t}) ]$. We call this three-tuple a \emph{cohomology datum} for short. We just write $0$ for the trivial cohomology datum $[0,0,0]$.
{ \tiny
\[
\left( 
\begin{array}{ccccccccccc}
[1,0,0] &  0  &   0  &   0  &   0 &   0  &   0   &   0   &   0  &    0 &    0\\
0  & [1,0,0] & 0 &  0  &  0 &   0  &  0 &  0 &   0  &  0  & 0 \\
 \, [0,1,0] &  [0,1,0] &  [1,0,0] & 0 &  0 &  0 &  0 &  0 &  0 &  0 & 0 \\
0 &  0 &  [0,0,1]  & [1,0,0] & 0 &   0  &  0 &  0 &  0 &  0 & 0  \\
0 &  0 &  [0,0,1] & 0 &  [1,0,0] &  0 &  0 &  0 &  0 &  0  & 0 \\
 0 &  0 &  [0,0,1] &  0 &  0 &  [1,0,0] & 0 &  0 &  0 &  0 & 0   \\
 \, [0,1,0] & [0,1,0] & 0 &   [0,1,0] & [0,1,0] & [0,1,0] & [1,0,0] & 0 &  0 &  0 & 0   \\
 0  & 0 &  [0,0,1] &  0 &   0 &   0 &  [0,0,1] &  [1,0,0] &  0 &   0 & 0   \\
 0 &   0 &   [0,0,1] &  0 &   0 &   0 &  [0,0,1] &  0 &  [1,0,0] &  0 & 0   \\
 0 & 0 & [0,0,1] &  0 & 0  & 0 & [0,0,1] & 0 & 0 & [1,0,0] & 0 \\
 \, \ast & \ast &   [0,0,1] & \ast &  \ast  & \ast &    [0,0,1] &  \ast & \ast & \ast &   [1,0, 0] 
\end{array} \right)
\]
}
Here the entry $\ast$ means either $[0,0,0]$ or $[0,1,1]$. 
\end{lemma}

\begin{proof}
We check this for the sequence $(\mathcal{L}_{1}, \dots , \mathcal{L}_{11})$ by explicit calculation. The general statement follows by upper-semicontinuity. 
\end{proof}

\begin{remark}\xlabel{rVisual}
It is helpful to think of the degree $0$ line bundles $\mathcal{L}_{i, t}$, $i\neq 3, 7$, as ``circles" and of the degree $1$ line bundles $\mathcal{L}_{3,t}$, $\mathcal{L}_{7, t}$ as ``squares". Then the assertion of Lemma \ref{lRigidity} can be paraphrased as saying that there are no derived homomorphisms between circles except from the first to the last (eleventh) circle where one can have the cohomology datum $[0,1,1]$. The squares are completely orthogonal, and the derived homomorphisms of a circle into a square (in the forward direction) have $\chi=-1$ and cohomology datum $[0,1,0]$, whereas derived homomorphisms of a square into a circle (in the forward direction) have $\chi=1$ and cohomology datum $[0,0,1]$. In the proof of Proposition \ref{pRigidity} we will first show that $H^* (\mathfrak{A}_t)$ has no higher multiplication and then that the algebra structure is also fixed in the neighbourhood of a generic point. It is instructive to think of pictures like the following illustrating a potential composition for $m_4$:
\vspace{1cm}
\[
\xymatrix{
\circ \ar@/^2pc/[rr]^{\chi = -1, \: [0,1,0]}  \ar@/_2pc/[rrrrrrrrrr]_{\chi =0, \: [0,1,1]\: \mathrm{or} \: [0,0,0]} & \circ   & \square \ar@/^2pc/[rr]^{\chi = 1, \: [0,0,1]} & \circ & \circ \ar@/^2pc/[rr]^{\chi = -1, \: [0,1,0]}& \circ & \square \ar@/^2pc/[rrrr]^{\chi = 1, \: [0,0,1]} & \circ & \circ & \circ & \circ 
}
\]
\vspace{1cm}
\end{remark}

\begin{proof}(of Proposition \ref{pRigidity}) 
We choose $t$ in a neighbourhood of a generic point in the moduli space of determinantal Barlow surfaces; hence we can assume that the matrix of cohomology data in Lemma \ref{lRigidity} is constant (i.e. there are no changes of the entries $\ast$ in the matrix). 

We think of the $\mathcal{L}_{i, t}$ as the objects of our $A_{\infty}$-category. It is clear that
\[
m_2 : \mathrm{hom} (X_0, X_1) \otimes \mathrm{hom} (X_1, X_2) \to \mathrm{hom} (X_0, X_2)
\]
is always the zero map in our case if $X_0$, $X_1$, $X_2$ are pairwise different; in fact, the only way to get a potentially nonzero composition would be to compose a morphism from a circle to a square with a morphism from that square to the last circle, which is impossible because this is a degree $3$ morphism; or to compose a morphism from a square to a circle with a morphism from that circle to the last circle, but this is also at least of degree $3$.

Hence it suffices to prove that there is no higher multiplication, i.e. $m_i = 0$ for $i \ge 3$. Then the endomorphism algebra of our category is just a usual graded algebra, and the algebra structure is completely determined and does not deform.

Clearly, $m_d=0$ for $d\ge 6$: in fact, if $i< j<k<l<m <n < o$, one of the spaces
\begin{gather*}
\mathrm{RHom}^{\bullet} (\mathcal{L}_{i, t},\;  \mathcal{L}_{j, t}), \mathrm{RHom}^{\bullet} (\mathcal{L}_{j, t}, \mathcal{L}_{k, t})\\
\mathrm{RHom}^{\bullet} (\mathcal{L}_{k, t},\;  \mathcal{L}_{l, t}), \mathrm{RHom}^{\bullet} (\mathcal{L}_{l, t}, \mathcal{L}_{m, t}), \; \mathrm{RHom}^{\bullet} (\mathcal{L}_{m, t}, \mathcal{L}_{n, t}), \; \mathrm{RHom}^{\bullet} (\mathcal{L}_{n, t}, \mathcal{L}_{o, t})
\end{gather*}
is the zero space. 

Now look at $m_5$: the smallest degree of a nonzero element in a space
\[
 \mathrm{hom} (\mathcal{L}_{i, t}, \mathcal{L}_{j, t}) \otimes \mathrm{hom} (\mathcal{L}_{j, t}, \mathcal{L}_{k, t}) \otimes \mathrm{hom} (\mathcal{L}_{k, t}, \mathcal{L}_{l, t}) \otimes \mathrm{hom} (\mathcal{L}_{l, t}, \mathcal{L}_{m, t}) \otimes \mathrm{hom} (\mathcal{L}_{m, t}, \mathcal{L}_{n, t} ) 
\]
for $i<j<k<l<m<n$ is equal to $7$. But $m_5$ lowers the degree by $3$, and there are no $\mathrm{Ext}^4$'s.

For $m_4$ resp. $m_3$ we argue similarly: the lowest degrees of nonzero elements in the spaces of $4$ resp. $3$ composable morphisms are $6$ resp. $4$; but $m_4$ lowers the degree by $2$ and $m_3$ lowers the degree by $1$. 
\end{proof}

We will need the following special case of a result by Voisin, see \cite{Voi}.
\begin{theorem}\xlabel{cBloch}
The generic determinantal Barlow surface $\tilde{X}$ of Section \ref{sConstruction} satisfies the Bloch conjecture $\mathrm{CH}^2 (\tilde{X}) = \ZZ$.
\end{theorem}

\begin{corollary}\xlabel{cPhantom}
The exceptional sequence $(\mathcal{L}_{1, t}, \dots , \mathcal{L}_{11, t})$ is not full, in other words, there exists a phantom category $\mathcal{A}_t$ in $\mathrm{D}^b (S_t)$ as in Theorem \ref{tMain} for a surface $S_t$ which is generic in a small neighbourhood of the Barlow surface $S = S_0$ in the moduli space of determinantal Barlow surfaces.
\end{corollary}

\begin{proof}
By Proposition \ref{pRigidity}, the subcategories $\langle \mathcal{L}_{1, t}, \dots , \mathcal{L}_{11, t}\rangle$ are all equivalent. However, by \cite{Kawamata02}, two minimal surfaces of general type whose derived categories are equivalent are isomorphic. It follows that the sequence $( \mathcal{L}_{1, t}, \dots , \mathcal{L}_{11, t})$ cannot be full, i.e. there is a nontrivial complement $\mathcal{A}_t$ (generically). Since $\mathrm{K}_0 (S_t)$ is isomorphic to $\ZZ^{11}$ it follows that $\mathcal{A}_t$ is a phantom. Here we use Theorem \ref{cBloch} for the generic determinantal Barlow surface to have $\mathrm{K}_0 (S_t)\simeq \ZZ^{11}$.
\end{proof}

\section{Heights of exceptional collections and a phantom on the Barlow surface}\xlabel{sKuznetsov}

This section contains a proof that the Barlow surface $S$ itself contains a phantom (whereas Corollary \ref{cPhantom} gives this for a general determinantal Barlow surface somewhere in the moduli space). For this, we will use a result of Kuznetsov concerning heights of exceptional collections, see \cite{Kuz12}. First we need to recall some notions.

Given two objects $F,F'$ in a triangulated category $\mathcal{T}$, we define their relative height to be
\[e(F,F'):=\min \left\{p\in \mathbb{Z}\;|\; \text{Hom}(F,F'[p])\neq 0\right\}.\]

Now consider an exceptional collection $(E_1,\ldots, E_n)$ in the bounded derived category of coherent sheaves on some smooth projective variety $Z$ and an autoequivalence $\Phi\colon \mathcal{T}\rightarrow \mathcal{T}$. The \emph{extended collection} (which need not be exceptional) is $(E_1,\ldots, E_n,\Phi(E_1),\ldots, \Phi(E_n))=(E_1,\ldots,E_{2n})$.

An \emph{arc} from $i$ to $j$, where $1\leq i<j\leq i+n\leq 2n$, is an increasing set of integers
\[i=a_0<a_1<\ldots <a_{k-1}<a_k=j,\]
and the \emph{length} of an arc is set to be $a_k-a_0$. An arc of the maximal possible length $n$ is called a \emph{circle}. The set of all possible circles is denoted by $\text{circ}(n)$. 

For an extended collection as above and an arc $a=(a_0,\ldots, a_k)$ the \emph{height} of the collection along $a$ is defined as
\[e_a:=e(E_{a_0},E_{a_1})+\ldots +e(E_{a_{k-1}},E_{a_k})-k+1.\]
The $\Phi$-\emph{height} of the collection is then
\[h_{\Phi}:=\min_{a \in \text{circ}(n)} e_a.\]

Very roughly speaking, the height measures the smallest degree one can get by composing elements in the Yoneda algebra associated to the exceptional collection.\smallskip

In the following we will be concerned with the case where $\Phi\cong -\otimes \omega_Z^{-1}$. The associated height is called the \emph{anticanonical height} and denoted by $h$. The following result is from \cite{Kuz12}.

\begin{proposition}
Let $(E_1,\ldots, E_n)$ be an exceptional collection in $\mathrm{D}^b(Z)$. If $h\geq 1-\dim(Z)$, then the collection is not full.
\end{proposition}

We can apply this to the Barlow surface.

\begin{proposition}
The height of the anticanonically extended collection constructed in Section \ref{sDeformation} is equal to $2$.
\end{proposition}

\begin{proof}

For $\sL_i, i=1,\dots,11$ as in Proposition \ref{pSequence} and $\sL_{i+11} := \sL_i(-\tildeK)$ we use Proposition \ref{pSectionEstimate}
to directly calculate lower bounds
\[
\bigl(e(\sL_j,\sL_i)\bigr)_{i=1\dots22, j=1\dots22} \ge
\] \[
	 {\tiny \left( 
	\begin{array}{c@{}c@{}c@{}c@{}c@{}c@{}c@{}c@{}c@{}c@{}c|c@{}c@{}c@{}c@{}c@{}c@{}c@{}c@{}c@{}c@{}c}
	- & - & - & - & - & - & - & - & - & - & - & - & - & - & - & - & - & - & - & - & - & - \\
	\infty & - & - & - & - & - & - & - & - & - & - & - & - & - & - & - & - & - & - & - & - & - \\
	1 &1 &- & - & - & - & - & - & - & - & - & - & - & - & - & - & - & - & - & - & - & - \\
	\infty & \infty & 2 &- & - & - & - & - & - & - & - & - & - & - & - & - & - & - & - & - & - & - \\
	\infty & \infty & 2 &\infty & - & - & - & - & - & - & - & - & - & - & - & - & - & - & - & - & - & - \\
	\infty & \infty & 2 &\infty & \infty & - & - & - & - & - & - & - & - & - & - & - & - & - & - & - & - & - \\
	1 &1 &\infty & 1 &1 &1 &- & - & - & - & - & - & - & - & - & - & - & - & - & - & - & - \\
	\infty & \infty & 2 &\infty & \infty & \infty & 2 &- & - & - & - & - & - & - & - & - & - & - & - & - & - & - \\
	\infty & \infty & 2 &\infty & \infty & \infty & 2 &\infty & - & - & - & - & - & - & - & - & - & - & - & - & - & - \\
	\infty & \infty & 2 &\infty & \infty & \infty & 2 &\infty & \infty & - & - & - & - & - & - & - & - & - & - & - & - & - \\
	1 &1 &2 &1 &\infty & \infty & 2 &\infty & \infty & 1 &- & - & - & - & - & - & - & - & - & - & - & - \\
	\hline
	2 &2 &2 &2 &2 &2 &2 &2 &2 &2 &2 &- & - & - & - & - & - & - & - & - & - & -  \\
	- & 2 &2 &2 &2 &2 &2 &2 &2 &2 &2 &\infty & - & - & - & - & - & - & - & - & - & - \\
	- & - & 2 &1 &\infty & \infty & 2 &\infty & \infty & 1 &\infty & 1 &1 &- & - & - & - & - & - & - & - & - \\
	- & - & - & 2 &2 &2 &2 &2 &2 &2 &2 &\infty & \infty & 2 &- & - & - & - & - & - & - & - \\
	- & - & - & - & 2 &2 &2 &2 &2 &2 &2 &\infty & \infty & 2 &\infty & - & - & - & - & - & - & - \\
	- & - & - & - & - & 2 &2 &2 &2 &2 &2 &\infty & \infty & 2 &\infty & \infty & - & - & - & - & - & - \\
	- & - & - & - & - & - & 2 &\infty & \infty & 1 &\infty & 1 &1 &\infty & 1 &1 &1 &- & - & - & - & - \\
	- & - & - & - & - & - & - & 2 &2 &2 &2 &\infty & \infty & 2 &\infty & \infty & \infty & 2 &- & - & - & - \\
	- & - & - & - & - & - & - & - & 2 &2 &2 &\infty & \infty & 2 &\infty & \infty & \infty & 2 &\infty & - & - & - \\
	- & - & - & - & - & - & - & - & - & 2 &2 &\infty & \infty & 2 &\infty & \infty & \infty & 2 &\infty & \infty & - & - \\
	- & - & - & - & - & - & - & - & - & - & 2 &1 &1 &2 &1 &\infty & \infty & 2 &\infty & \infty & 1 &- 
	\end{array}
	\right) }
\]
where we placed $-$'s where no information is needed. By results of \cite{Kuz12} the height is at least $2$. 
Moreover, it is clear by consideration of the first column that the height is at most $2$. 

One can also prove this result by hand as follows. First note that a circle with one segment has height at most 2. Indeed, we have to consider $\mathrm{Hom}(L_i,(L_{i+11}-\tildeK)[p])\cong H^{2-p}(2\tildeK)$ and this is $0$ for $p\leq 2$ by Kodaira-Viehweg vanishing. On the other hand, $H^0(2\tildeK)\neq 0$, so $h\leq 2$.

Thus, we consider circles with at least 2 segments. By \cite{Kuz12} we can assume that $a_{k-1}\leq 11$. The fact that $h\geq 0$ then immediately follows from the data collected in the matrix in Lemma \ref{lRigidity}, since any of the first $k-1$ segments of a circle $a$ contributes at least $1$ to the sum $e_a$. The last segment's contribution is at least $0$, because all members of the collection are sheaves.

In fact, one can prove that $h\geq 2$ as follows. We will concentrate on the last segment of a circle. Hence, the first object of this segment is in the original sequence and the second in the extended part. Also note that, by definition, the distance between them is at most 11. Since the canonical bundle is big and nef on the Barlow surface, the following statement holds.

If $L.K_S\geq L'.K_S$, then either $\Hom(L,L')=0$, $L\cong L'$ or $D=L'-L$ is a sum of $(-2)$-curves such that $D.K_S=0$. This follows at once from the assumption and the fact that a non-trivial homomorphism from $L$ to $L'$ gives a section of the effective divisor $L'-L$. Applying this to our extended sequence, we see that both line bundles involved in the segment have to be of degree 0, hence the second bundle in the segment can be either $E_{14}=L_3-K$ or $E_{18}=L_7-K$. Since the length of the segment is at most 11, we only have to consider the spaces $\Hom(L_i,L_3-K)$ for $i\geq 4, i\neq 7$ and $\Hom(L_j,L_7-K)$ for $j\geq 8$. Now any circle with $k$ segments whose last segment is $(L_i,L_3-K)$ with $i\geq 4, i\neq 7,11$, has the property that each of the first $k-1$ segments has relative height at least $2$ which, yet again, follows from Lemma \ref{lRigidity}. Hence, any such circle has length at least $1$. Same reasoning holds for circles involving $(L_j,L_7-K)$ with $8\leq j<11$.
Thus, we only need to consider the spaces $\Hom(L_{11},L_3-K)$ and $\Hom(L_{11},L_7-K)$ or, to put it differently, we have to check whether the divisors $L_3-K-L_{11}$ and $L_7-K-L_{11}$ are sums of irreducible effective $(-2)$-curves. Since the first divisor is $\widetilde{E}_2-\widetilde{E}^+_4$ and the second is $\widetilde{E}_2-\widetilde{E}^+_5$, their intersection with $\widetilde{C}^{\pm}_1$ is $0$ and the intersection with $\widetilde{C}^{\pm}_2$ is $1$. Hence, they cannot be sums of the exceptional curves. Furthermore, one can directly check, that, in fact, $\Hom(L_{11},(L_7-\tildeK)[1])=0$, so the above argument readily gives $h\geq 2$.

\end{proof}

\section{Conjectures} \xlabel{sConjectures}

Recently some progress was made towards proving Kontsevich's Homological Mirror Symmetry (HMS) conjecture, originally proposed for Calabi--Yau, and a little later also for Fano varieties - see \cite{AKO}. The scope of HMS was then extended to varieties of general type in  \cite{KKOY} and \cite{GKR}. As noted ibid.\
we  have a behaviour of Fukaya categories similar to derived categories for manifolds of general type due to the fact that Landau--Ginzburg models determine their geometries. So we propose

\begin{conjecture} 
The Fukaya category of the Barlow surface associated to a generic K\"ahler form has an orthogonal decomposition into a phantom category and eleven matrix factorization categories of type $\mathrm{MF}(\mathbb{C}(x,y)/x^2+ y^2)$.
\end{conjecture} 

This observation suggests that a geometric procedure which might lead to creation of phantoms is a series of rational  blow downs and smoothings - something known as conifold and extremal exoflop transitions in physics, see \cite{WIT} and \cite{ASP}. The variety of phantoms one might get seems very big as \cite{AV} suggests. As a result we put forward the following
 
\begin{conjecture}[see also  \cite{FHK}] 
There exist exotic symplectic manifolds homeomorphic to Del
  Pezzo surfaces, which can be distinguished by phantom orthogonal
  components in their Fukaya categories.
\end{conjecture}

When applied to dimension 3, the above observation leads to new conjectural nonrationality criteria for conic bundles.

\begin{conjecture}[see \cite{FHK}, \cite{CKP}] Let $X$ be a conic bundle
  of dimension 3.  If $D^b(X) = (\mathcal A, E_1,\ldots , E_n) $
  is a semiorthogonal decomposition into a phantom $\mathcal A$ and an
  exceptional sequence, then $X$ is not rational.
\end{conjecture}

\end{document}